\documentclass[12pt]{article}
\usepackage{graphicx}
\usepackage{caption}
\usepackage{csquotes}
\usepackage[english]{babel}
\usepackage{amsthm}
\usepackage{subcaption}
\usepackage[colorlinks, citecolor=blue, urlcolor = blue]{hyperref}
\usepackage{amssymb}
\usepackage{amsmath}
\usepackage{bm}
\usepackage{comment}
\usepackage[shortlabels]{enumitem}
\usepackage{multicol}
\theoremstyle{plain}
\newtheorem{theorem}{Theorem}[section]

\theoremstyle{definition}
\newtheorem{definition}[theorem]{Definition}

\theoremstyle{plain}

\theoremstyle{plain}
\newtheorem{corollary}[theorem]{Corollary}

\theoremstyle{plain}
\newtheorem{lemma}[theorem]{Lemma}

\theoremstyle{definition}
\newtheorem{example}[theorem]{Example}

\theoremstyle{definition}
\newtheorem{fact}[theorem]{Fact}

\renewcommand\bar{\overline}

\newcommand{\sube}{\subseteq}

\newcommand{\RR}{\mathbb{R}}

\newcommand{\NN}{\mathbb{N}}

\newcommand{\BB}{\mathbb{B}}

\newcommand{\gph}{\operatorname{gph}}

\newcommand{\dom}{\operatorname{dom}}
\newcommand{\inte}{\operatorname{int}}

\title{On a Non-Prox-Regularity Example by Rockafellar and Wets}
\author{Isaac Jasper\thanks{Department of Mathematics, Irving K.\ Barber Faculty of Science, University
of British Columbia, Kelowna, British Columbia V1V 1V7,
Canada. E-mail: \texttt{ijasper@student.ubc.ca}} and Xianfu Wang\thanks{Department of Mathematics, Irving K.\ Barber Faculty of Science, University
of British Columbia, Kelowna, British Columbia V1V 1V7,
Canada. E-mail: \texttt{shawn.wang@ubc.ca}}}

\begin{document}

\maketitle
\begin{abstract} \noindent
We analyze a counterexample concerning non-prox-regular functions provided by Rockafellar and Wets in their 1998 foundational monograph, \textit{Variational Analysis}. By establishing a general theorem on the continuity of set-valued mappings, we show that the counterexample does not have its intended properties. We propose a slight adjustment to the construction and prove that the corrected version serves as the desired counterexample.
\end{abstract}



\section{Introduction}
Prox-regular functions are important in variational analysis and optimization; see, e.g.,
\cite{Poliquin}, \cite{poliquin1}, \cite[Chapter 13]{Rockafellar} for finite-dimensional spaces, and \cite{lionel1}, \cite{BERNARD20051}, \cite{lionel2}, \cite{boris24} for infinite-dimensional spaces.
To state the goal of this paper, we start with the following fundamental concepts from variational analysis.
\begin{definition}
        (Moreau envelopes and proximal mappings). For a proper, lsc function $f: \RR^n \to \bar \RR$ and $\lambda > 0$, the Moreau envelope function $e_\lambda f$ and proximal mapping $P_\lambda f$ are defined by
        \begin{align*}
        e_\lambda f(x) &:= \inf_w \left\{ f(w) + \frac{1}{2\lambda} \|w - x\|^2 \right \} \le f(x), \\
        P_\lambda f(x) &:= \arg \min_w \left\{ f(w) + \frac{1}{2\lambda} \|w-x\|^2\right\}.
        \end{align*}
    \end{definition}
Observe that while $P_{\lambda}f$ is a Lipschitz mapping with Lipschitz modulus $1$ for a convex function $f$ \cite{bauschke}, in general $P_\lambda f$ is a set-valued mapping \cite{Rockafellar}.
For a prox-regular function,
   Rockafellar and Wets proved the following beautiful result; see \cite[Proposition 13.37]{Rockafellar} and
also \cite[Theorem 4.4]{Poliquin}.
    \begin{fact}
    \label{_rockafellarProp}
    Suppose that $f : \RR^n \to \bar \RR$ is prox-regular at $\bar x$ for $\bar v = 0$, and that $f$ is prox-bounded. Then for all $\lambda > 0$ sufficiently small, there is a neighborhood of $\bar x$ on which
    \begin{enumerate}[(i)]
        \item\label{i:mor1}
        $P_\lambda f$ is monotone, single-valued and Lipschitz continuous; $P_\lambda f(\bar x) = \bar x$,
        \item\label{i:mor2}
         $e_\lambda f$ is differentiable with $\nabla (e_\lambda f)(\bar x) = 0$, in fact of class ${\cal C}^{1+}$ with
        \[
        \nabla e_\lambda f = \lambda^{-1}[I - P_\lambda f] = [\lambda I + T^{-1}]^{-1}
        \]
        for an $f$-attentive localization $T$ of $\partial f$ at $(\bar x, 0)$.
    \end{enumerate}
    \end{fact}
 Rockafellar and Wets \cite[page 618]{Rockafellar} then claimed that the function
\begin{equation}\label{e:rock-wets}
    f:\RR\rightarrow\RR: x\mapsto
    \begin{cases}
        |x| \left( 1 + \sin \left( \frac 1 x \right) \right), & x \ne 0 \\
        0, & x = 0
    \end{cases}
\end{equation}
\noindent is an example showing that Fact~\ref{_rockafellarProp}\ref{i:mor1} and \ref{i:mor2}
 do not suffice for prox-regularity of $f$ at $\bar x =0$ for $\bar v = 0$. See Figure~\ref{fig:fPlot} below for the plot of the function.
 \begin{figure}[h]
        \centering
        \includegraphics[width=0.8\linewidth]{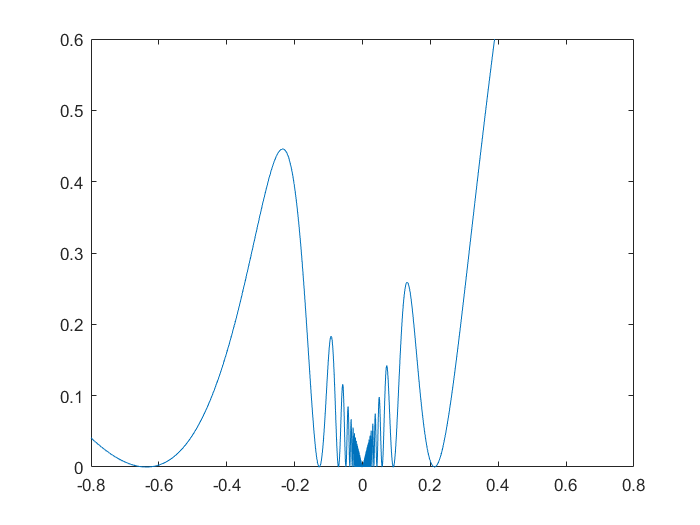}
        \caption{Plot of $f$.}
        \label{fig:fPlot}
    \end{figure}
The same results has also been documented in \cite[Remark 5.4]{BERNARD20051}.
    Indeed, it is simple to show that $f$ is not prox-regular at $\bar x = 0$ for $\bar v = 0$ and that $P_\lambda f(x)$ is monotone. But the other properties are more difficult to show. It turns out,
    for all neighborhoods of $0$, $P_\lambda f$ is not continuous, much less Lipschitz continuous,
    and has points where it is multivalued in all neighborhoods of $0$. This is one of the core results
    of this paper. Additionally, at the end of this paper we show that a small change to $f$ results in a function that indeed has the claimed properties.

The rest of the paper is organized as follows. Section~\ref{s:basic} presents the prerequisite definitions and results in variational analysis. In section~\ref{sect:proxExample} we present an abstract theorem regarding the continuity of set-valued mappings and use it to show that the counter example proposed by Rockafellar and Wets does not work as intended. In section~\ref{s:final} we provide a family of functions which have locally Lipschitz proximal mappings but the functions are not prox-regular at $0$, which in turn  give a corrected version of the counter example that works as Rockafellar and Wets originally intended.

\textbf{Setup and notation:} Our notations follow \cite{Rockafellar}. Throughout $\RR^n$ denotes the standard Euclidean space with inner product $\langle x,y \rangle := \sum_{i=1}^n x_iy_i$ and norm $\|x\| := \sqrt{\langle x, x\rangle}$ for $x,y \in \RR^n$. The extended real line is $\bar \RR := \RR \cup \{+\infty, -\infty\}$. An open ball of radius $\varepsilon$ around a point $\bar x \in \RR^n$ is denoted by $\BB_\varepsilon(\bar x) := \{x\in\RR^n\mid \|x - \bar x\| < \varepsilon\}$ and the unit ball we denote by $\BB := \BB_1(0)$.
A set $C \sube \RR^n$ is convex if for any $x_0,x_1 \in C$ and any $\lambda \in [0,1]$ we have
$(1-\lambda) x_0 + \lambda x_1 \in C.$
 A function $f: \RR^n \to \bar \RR$ is proper if $f(\bar x) < +\infty$ for some $\bar x\in \RR^n$, and $f(x) > -\infty$ for all $x\in \RR^n$. For a sequence $\{x^\nu \}_{\nu \in \NN}$ in $\RR^n$ we write $x^\nu \xrightarrow[f]{} \bar x$ if $\{x^\nu\}_{\nu\in\NN}$ converges to $\bar x$ and $\{f(x^\nu)\}_{\nu \in \NN}$ converges to $f(\bar x)$. For a set $A \sube \RR^n$ we write $x^\nu \xrightarrow[A]{} \bar x$ when $\{x^\nu\}_{\nu\in\NN}$ converges to $\bar x$ while each $x^\nu \in A$.
We say that $f:\RR^n\rightarrow\bar\RR$ is ${\cal C}^{1+}$ if it has a locally Lipschitz continuous derivative.
By $S:\RR^n\rightrightarrows \RR^m$ we mean a set-valued mapping from $\RR^n$ to subsets of $\RR^m$. We say that $S$ is empty-valued, single-valued, or multivalued at $x$ according to whether $S(x)$ is the empty set, a singleton, or a set containing more than one element. $S$ is closed valued or convex valued if the sets $S(x)$ are closed or convex respectively.
Moreover, we write
$\gph S := \{ (x,u)\in \RR^n\times \RR^m\mid u \in S(x) \}$ for the graph of $S$, and $\dom S := \{x\in\RR^n \mid S(x) \ne \varnothing\}$ for the domain of $S$. Finally, ${\mathcal N}_\infty$ denotes subsets of $\NN$ containing all $\nu$ sufficiently large.

\section{Basic definitions and preliminaries}\label{s:basic}
	In this section, for the readers’ convenience, we collect some definitions and basic facts regarding prox-regularity, proximal mappings and set-valued mappings, which will be used in sequel.
More details can be found in \cite{Rockafellar}.

    \begin{definition}
        A function $f: \RR^n \to \bar \RR$ is lower semicontinuous (lsc) at $\bar x$ if
        $
        \liminf_{x \to \bar x} f(x) \ge f(\bar x),$ and lsc on $\RR^n$ if this holds for every $\bar x\in\RR^n$.
    \end{definition}

Subdifferentials are fundamental in variational analysis, see, e.g., \cite{Rockafellar}, \cite{boris06}.

    \begin{definition} (regular and general subgradients).
         For a function $f: \RR^n \to \bar \RR$ at a point $\bar x$ with $f(\bar x)$ finite, a vector $v \in \RR^n$ is a regular subgradient of $f$ at $\bar x$, written $v\in \hat\partial f(\bar x)$, if
        \[
        \liminf_{\substack{x \to \bar x \\ x \ne \bar x}} \frac{f(x) - f(\bar x) - \langle v, x - \bar x\rangle}{\|x - \bar x \|} \ge 0;
        \]
        $v$ is a general subgradient (Mordukhovich limiting subgradient) of $f$ at $\bar x$, written $v \in \partial f(\bar x)$, if there exists a sequence $x^\nu \xrightarrow[f]{} x$ and $v^\nu \to v$ with $v^\nu \in \hat \partial f(x^\nu)$. An $f$-attentive localization of $\partial f$ around $(\bar x, \bar v)$ is some set-valued mapping $T$ which has a graph containing all points $(x,v) \in \gph \partial f$ with $\| v- \bar v\| < \varepsilon$, $\|x - \bar x\| < \varepsilon$ and $f(x) < f(\bar x) + \varepsilon$ for some $\varepsilon > 0$.
    \end{definition}
\noindent We say that a function $f$ is differentiable at a point
$\bar x$ if there is some vector $v \in \RR^n$ such that
    \[
    f(x) = f(\bar x) + \langle v, x - \bar x \rangle + o(\|x-\bar x\|)
    \]
    where $o(t)$ indicates a term such that
$
    \tfrac{o(t)}{t} \to 0 \text{ as } t \to 0, t \ne 0$. Write $v=\nabla f(\bar x)$.

    \begin{fact} \label{fact:diffSubdifEq} \cite[Exercise 8.8] {Rockafellar}.
        If $f: \RR^n \to \bar \RR$ is differentiable at $\bar x$ then $\hat \partial f(\bar x)= \{\nabla f(\bar x)\}$.
    \end{fact}

The key concept we need is the prox-regularity of a function.
    \begin{definition} (prox-regularity of functions).
        A function $f: \RR^n \to \bar \RR$ is prox-regular at $\bar x$ for $\bar v$ if $f$ is finite and locally lsc at $\bar x$ with $\bar v \in \partial f(\bar x)$, and there exists $\varepsilon > 0$ and $\rho \ge 0$ such that
        \begin{align*}
            f(x') \ge f(x) + \langle v, x' - x\rangle - \frac{\rho}{2} \|x' - x \|^2 \text{ for all } x' \in \BB_\varepsilon(\bar x) \\
            \text{when } v \in \partial f(x), \|v - \bar v\| < \varepsilon, \|x - \bar x \|< \varepsilon,
            f(x) < f(\bar x) + \varepsilon.
        \end{align*}
        When this holds for all $\bar v \in \partial f(\bar x)$, $f$ is said to be prox-regular at $\bar x$.
    \end{definition}

To make the Moreau envelope of a possibly nonconvex function well-defined, we require an associated definition.
   \begin{definition} (prox-bounded).
        A function $f:\RR^n \to \bar \RR$ is prox-bounded if there exists $\lambda > 0$ such that $e_\lambda f(x) > -\infty$ for some $x \in \RR^n$. The supremum of all such $\lambda$ is called the threshold $\lambda_f$ of prox-boundedness for $f$.
    \end{definition}

    \begin{fact} \label{fact:proxMapNonemptyCompact} \cite[Theorem 1.25]{Rockafellar}.
        Let $f: \RR^n \to \bar \RR$ be proper, lsc, and prox-bounded with threshold $\lambda_f > 0$. Then for any $\lambda \in (0, \lambda_f)$ and $x \in \RR^n$ the set $P_\lambda f(x)$ is nonempty and compact, and $e_\lambda f(x)$ is finite and depends continuously on $(\lambda, x)$ with $e_{\lambda}f(x)\uparrow f(x)$ for all $x$ as $\lambda\downarrow 0$. In particular, $P_\lambda f$ has a full domain, i.e., $\dom P_\lambda f = \RR^n$.
    \end{fact}

Monotone mappings are ubiquitous in modern optimization and variational analysis; see, e.g., \cite{bauschke}, \cite{Rockafellar}.
       \begin{definition} (monotone mappings).
        A mapping $S : \RR^n \rightrightarrows \RR^n$ is monotone if, given any $x_0, x_1 \in \RR^n$, all $y_0 \in S(x_0)$, and all $y_1 \in S(x_1)$ we have
        $
        \langle y_1 - y_0, x_1 - x_0\rangle \ge 0
      $
    \end{definition}

   \begin{fact}  \label{fact:proxMonotone} \cite[Proposition 12.19]{Rockafellar}.
        For a proper, lsc function $f: \RR^n \to \bar \RR$ and any $\lambda > 0$, the proximal mapping $P_\lambda f:\RR^n \rightrightarrows \RR^n$ is monotone.
    \end{fact}

    \begin{definition} (local boundedness).
        A mapping $S : \RR^n \rightrightarrows \RR^m$ is locally bounded on $\RR^n$ if it has the property that whenever $u^\nu \in S(x^\nu)$ and the sequence $\{x^\nu\}_{\nu \in \NN}$ is bounded then the sequence $\{u^\nu\}_{\nu \in \NN}$ is bounded.
    \end{definition}

    \begin{definition} (semicontinuity of set-valued mappings).
   For a set-valued mapping $S:\RR^n\rightrightarrows\RR^m$,  define its outer and inner limits at $\bar x$ as follows:
    \begin{align*}
        \limsup_{x \to \bar x} S(x)& := \left\{ u \mid \exists x^\nu \to \bar x, \exists u^\nu \to u \text{ with } u^\nu \in S(x^\nu) \right \}, \\
        \liminf_{x \to \bar x} S(x)& := \left\{ u \mid \forall x^\nu \to \bar x, \exists N \in {\mathcal N}_\infty ,\exists u^\nu \xrightarrow[]{} u \text{ with } \nu \in N, u^\nu \in S(x^\nu) \right \}.
    \end{align*}
    We say that $S$ is outer semicontinuous (osc) at $\bar x$ if
    $
    \limsup_{x \to \bar x} S(x) \sube S(\bar x),
    $
    inner semicontinuous (isc) at $\bar x$ if
    $
    \liminf_{x \to \bar x} S(x) \supseteq S(\bar x),
    $
    and  is continuous at $\bar x$ if both conditions hold.
    \end{definition}

    \begin{fact} \label{fact:oscIffGphClosed} \cite[Theorem 5.7(a)]{Rockafellar}.
        A mapping $S: \RR^n \rightrightarrows \RR^m$ is everywhere osc if and only if $\gph S$ is closed in $\RR^n\times\RR^m$.
    \end{fact}

    \begin{fact} \label{fact:proxOscLB} \cite[Example 5.23]{Rockafellar}.
        For any proper, lsc function $f: \RR^n \to \bar \RR$ that is prox-bounded with threshold $\lambda_f>0$
        and $\lambda \in (0, \lambda_f)$, its proximal mapping $P_{\lambda}f$ is everywhere osc and locally bounded.
    \end{fact}

         \begin{definition} (Lipschitz set-valued mappings). \label{d:lipset}
        A mapping $S : \RR^n \rightrightarrows \RR^n$ is Lipschitz continuous on $X \sube \RR^n$ if it is nonempty-closed-valued on $X$ and there exists $\kappa > 0$ such that
        \[
        S(x') \sube S(x) + \kappa \|x ' - x\| \BB
        \]
        for all $x,x' \in X$.
    \end{definition}

    It is not hard to show that if a mapping $S: \RR^n \rightrightarrows \RR^n$ is Lipschitz continuous then it is continuous.  For a monotone mapping with non-empty values on an open set,
    it turns out that its Lipschitz continuity in the sense of
    Definition~\ref{d:lipset} forces the mapping to be single-valued, thus, reduces to the Lipschitz continuity
    of a single-valued mapping.

    \begin{lemma}\label{l:singlevalue}
     Suppose that $S:\RR^n\rightrightarrows\RR^n$ is monotone and $\bar x\in\inte\dom S$.
    If there exists $\varepsilon>0$ such that $S$ is Lipschitz on $\BB_{\varepsilon}(\bar x)\subset\inte\dom S$
    in the sense of Definition~\ref{d:lipset},
    then $S$ is single-valued on $\BB_{\varepsilon}(\bar x)$.
    \end{lemma}
    \begin{proof}
  Since $S$ is monotone, by \cite[Theorem 21.27]{bauschke} there exists a dense set $C$ of $\BB_{\varepsilon}(\bar x)$ such that
  $S(x)$ is a singleton for every $x\in C$. Let $x\in \BB_{\varepsilon}(\bar x)$.
  Then there exists a sequence $\{x^{\nu}\}_{\nu\in\NN}$ from $C$ such that $S(x^{\nu})$ is a singleton
  and $\{x^{\nu}\}_{\nu\in\NN}$ converges to $x$. By the Lipschitz continuity of $S$, there exists $\kappa>0$
  such that
  $S(x)\subset S(x^{\nu})+\kappa\|x-x^{\nu}\|\BB$, which implies
  $\mbox{diam}(S(x))\leq 2\kappa\|x-x^{\nu}\|$. Taking limit as $\nu\rightarrow\infty$ gives $\mbox{diam}(S(x))=0$, so
  $S(x)$ is a singleton. Since $x \in \BB_{\varepsilon}(\bar x)$ was arbitrary, we conclude that
  $S$ is single-valued on $\BB_{\varepsilon}(\bar x)$.
    \end{proof}
 \begin{corollary}
 Let $f: \RR^n \to \bar \RR$ be proper, lsc, and prox-bounded with threshold $\lambda_f>0$, and let
 $\lambda \in (0, \lambda_f)$. If its proximal mapping $P_{\lambda}f$ is Lipschitz on an open set
 in the sense of Definition~\ref{d:lipset}, then $P_{\lambda}f$ is single-valued on the open set,
 thus Lipschitz in the usual sense.
 \end{corollary}
\begin{proof}
Combine Facts~\ref{fact:proxMapNonemptyCompact}, \ref{fact:proxMonotone}, and Lemma~\ref{l:singlevalue}.
\end{proof}

\section{Analysis of proximal mapping of Rockafellar-Wets example} \label{sect:proxExample}
  We begin with a general result about inner semicontinuity (isc) of a set-valued mapping, which is of independent
  interest.

    \begin{theorem}\label{_NecessaryConContGen}
        Let $S : \RR^n \rightrightarrows \RR^n$ be osc, locally bounded, and have a convex domain. Suppose there exist nonempty subsets of $\gph S$, $A$ and $B$, which are disjoint from each other's closure, that is $A \cap \bar B = \bar A \cap B = \varnothing$, and a point, $x_0 \in \dom S$, such that $\gph S = A \cup B$ and $\{(x_0,v) \mid v \in S(x_0)\} \sube A$. Then $S$ is not isc. Moreover, there exists a point where $S$ is multivalued.
    \end{theorem}
    \begin{proof}
        Observe that, by Fact~\ref{fact:oscIffGphClosed}, $\gph S$ is closed because $S$ is osc. If $x^\nu \xrightarrow[A]{} x$ then $x \in \gph S$ so $x \in A$ or $x \in B$ and since $\bar A \cap B = \varnothing$ we must have $x \in A$ so $A$ is closed. Similarly, $B$ is closed.

        Take $x_1$ such that $\{(x_1,v) \mid v \in S(x_1)\} \cap B \ne \varnothing$. $x_1$ exists because $B$ is a nonempty subset of $\gph S$. Define
        \[
        x_\gamma := (1-\gamma)x_0 + \gamma x_1 \quad \gamma \in [0,1].
        \]
        Observe that $x_\gamma \in \dom S$ for all $\lambda \in [0,1]$ because $\dom S$ is convex. Take
        \[
        \bar \lambda = \sup\Big\{\lambda\in [0,1]  \mid \forall \gamma \in [0, \lambda], \{ (x_\gamma, v) \mid v \in S(x_\gamma) \} \sube A \Big\}.
        \]
        Such a $\bar \lambda$ is finite because $\lambda = 0$ is in the set. Take the sequence $\lambda^\nu_- \to \bar \lambda$ with $\lambda^\nu_- \le \bar \lambda$. Then $x_{\lambda^\nu_-}$ converges to $x_{\bar \lambda}$ since $x_\lambda$ is continuous relative to $\lambda$. Since $x_{\lambda^\nu_-}$ converges it is bounded. So, by local boundedness, any sequence $y^\nu \in S(x_{\lambda^\nu_-})$ is bounded. Take a subsequence if necessary so that $y^\nu$ converges to some $\bar y$. Now, $(x_{\lambda^\nu_-}, y^\nu) \xrightarrow[{A}]{} (x_{\bar \lambda}, \bar y)$ and because $A$ is closed we have $(x_{\bar \lambda}, \bar y) \in A$. Now, by the way $\bar \lambda$ is defined, there exists a sequence $\lambda^\nu_+ \to \bar \lambda$ with $\lambda^\nu_+ \ge \bar \lambda$ such that
        \[
        \{ (x_{\lambda^\nu_+}, v) \mid v \in S(x_{\lambda_+^\nu}) \} \not \sube A.
        \]
        But, because $A$ and $B$ cover $\gph S$, we must have
        \[
        \{ (x_{\lambda^\nu_+}, v) \mid v \in S(x_{\lambda_+^\nu}) \} \cap B \ne \varnothing.
        \]
        Take $u^\nu \in S(x_{\lambda^\nu_+})$ such that $(x_{\lambda^\nu_+}, u^\nu) \in B$. Now, by similar argument as for $x_{\lambda^\nu_-}$, we have $(x_{\lambda^\nu_+}, u^\nu) \xrightarrow[B]{} (x_{\bar \lambda}, \bar u)$ and because $B$ is closed, $(x_{\bar \lambda}, \bar u) \in B$. Recall $S$ is isc at $\bar x$ if and only if
        \[
            \liminf_{x \to \bar x} S(x) \supseteq S(\bar x)
        \]
        with
        \[
        \liminf_{x \to \bar x} S(x) = \left \{ u \mid \forall x^\nu \to \bar x, \exists N \in {\cal N_\infty}, u^\nu \xrightarrow[N]{} u, \text{ with } u^\nu \in S(x^\nu) \right\}.
        \]
        Let $\bar x = x_{\bar \lambda}$. Now, for any choice of $y^\nu \to \bar y$ with $y^\nu \in S(x_{\lambda^\nu_-})$, we have $(x_{\bar \lambda}, \bar y) \in A$ and there exists $u^\nu \to \bar u$ with $u^\nu \in S(x_{\lambda^\nu_+})$ such that $(x_{\bar \lambda}, \bar u) \in B$. $A$ and $B$ are disjoint so we must have $\bar y \ne \bar u$. So $\bar u \not \in \liminf_{x \to x_{\bar \lambda}} S(x_{\bar \lambda})$ but $(x_{\bar \lambda}, \bar u) \in B \sube\gph S$ giving $\bar u \in S(x_{\bar \lambda})$. Therefore, $S$ is not isc at $x_{\bar \lambda}$. Moreover, for any given choice of $y^\nu \in S(x_{\lambda^\nu_-})$ with $y^\nu \to \bar y$, we have $\{\bar y, \bar u\} \sube S(x_{\bar \lambda})$ with $\bar y \ne \bar u$ so $S$ is multivalued at $x_{\bar \lambda}$.
    \end{proof}
    \begin{corollary}\label{_NecessaryConContRR}
        Let $S : \RR^n \rightrightarrows \RR^n$ be osc, locally bounded, and have convex domain. Suppose there exist nonempty subsets of $\gph S$, $A$ and $B$, which are disjoint from each other's closure, and points $x_1, x_0 \in \dom S$, such that $\gph S = A \cup B$, $\{(x_0, v) \mid v \in S(x_0)\} \sube A$, and $\{(x_1, v) \mid v \in S(x_1)\} \cap B \ne \varnothing$. Then, there exists a point $\bar x$ on the line segment connecting $x_0$ and $x_1$ such that $S$ is multivalued and not isc at $\bar x$.
    \end{corollary}
    \begin{proof}
        Take $\bar x = (1-\bar \lambda)x_0 + \bar \lambda x_1$ with
        \begin{align*}
        \bar \lambda &= \sup\Big\{\lambda\in [0,1]  \mid \forall \gamma \in [0, \lambda], \{ (x_\gamma, v) \mid v \in S(x_\gamma) \} \sube A \Big\}, \\
        \text{ and }x_\gamma & := (1-\gamma)x_0 + \gamma x_1.
        \end{align*}
        Then $\bar x$ is on the line segment connecting $x_0$ and $x_1$ and the proof of Theorem \ref{_NecessaryConContGen} shows that $S$ is multivalued and not isc at $\bar x$.
    \end{proof}

Let us now analyze the proximal mapping $P_{\lambda}f$ of the Rockafellar-Wets function $f$ given by \eqref{e:rock-wets}.
 Our goal is to show that
  in every neighborhood of $0$, there exists a point $x$ such that $P_{\lambda}f(x)$ is multivalued.
Since $f$ is continuous and bounded below,
 we have $\lambda_{f}=+\infty$.  By Fact~\ref{fact:proxOscLB} and Fact~\ref{fact:proxMapNonemptyCompact},
 for every $\lambda>0$,
   $P_\lambda f : \RR \rightrightarrows \RR$ is osc, locally bounded, and has a full domain.
  Therefore, in any neighborhood of $0$ it suffices to find a separation of $\gph P_\lambda f$ such that there are $x_0$ and $x_1$ in the neighborhood that are as defined in the requirements of Corollary~\ref{_NecessaryConContRR}.

    \begin{figure}[h]
        \centering
        \includegraphics[width=0.8\linewidth]{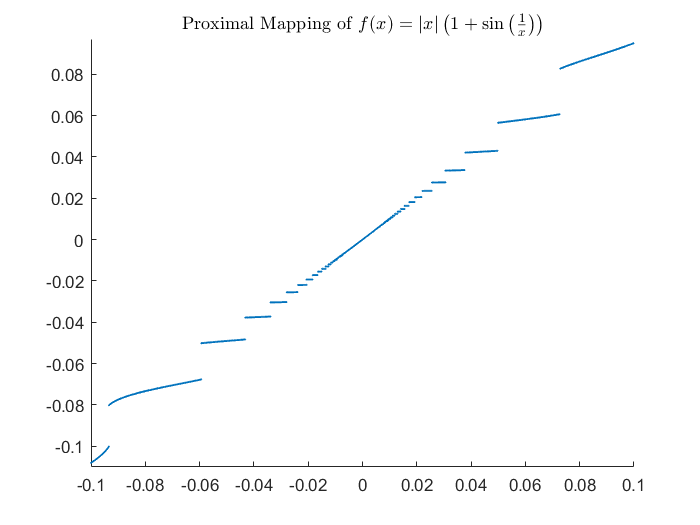}
        \caption{Plot of $P_\lambda f$ about the origin with $\lambda = 2^{-10}$. Observe the gaps in the graph.}
        \label{fig:proxMapAboutZero}
    \end{figure}

    To these ends we define the following two sequences.
        \[
    \ell_n = \frac{2}{\pi\left(4n + 3 \right)}, \quad m_n = \frac{2}{\pi\left(4 n + 1\right)} \quad \text{ for } n \in \NN.
    \]

    \begin{figure}[h]
        \centering
        \includegraphics[width=0.8\linewidth]{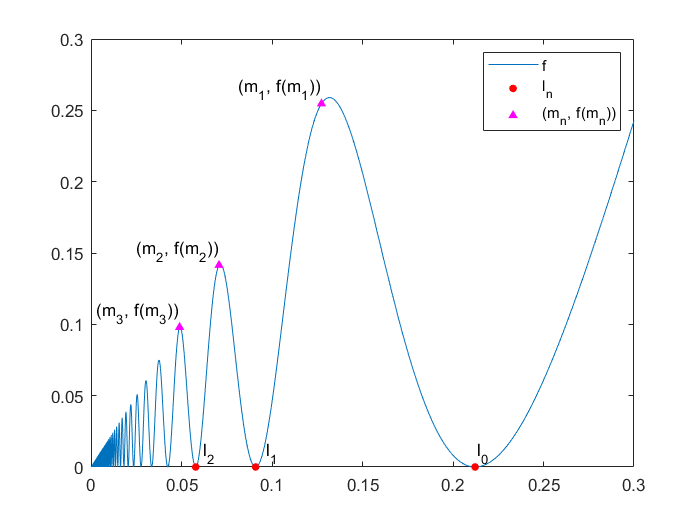}
        \caption{Plot of $f$ with selected values of $\ell_n$ and $m_n$ labeled.}
        \label{fig:fPlotLabled}
    \end{figure}

    \begin{lemma}
        $\ell_n$ and $m_n$ have the following properties:
        \begin{enumerate}[(i)]
            \item $\lim_{n \to \infty} \ell_n = \lim_{n \to \infty} m_n = 0$.

            \item $f(\ell_n) = 0$.

            \item $f(m_n) = 2m_n$.

            \item For all $n \in \NN$, $P_\lambda f(\ell_n) = \{\ell_n\}$.
        \end{enumerate}
    \end{lemma}

    \begin{proof}
        Property (i) is clear. For property (ii) observe that
        \begin{align*}
            f(\ell_n) &= |\ell_n| \left( 1 + \sin \left( \frac 1 \ell_n \right) \right) \\
            & = |\ell_n| \left( 1 + \sin \left( \frac{3 \pi }{2} + 2 \pi n \right) \right) = |\ell_n| (1 - 1) = 0.
        \end{align*}
        For property (iii) observe that
        \begin{align*}
            f(m_n) &= |m_n| \left( 1 + \sin \left( \frac 1 m_n \right) \right) \\
            & = m_n \left( 1 + \sin \left( \frac{\pi }{2} + 2 \pi n \right) \right)= 2 m_n.
        \end{align*}
        For property (iv) observe that since $\sin \left(\frac 1 x\right) \ge -1$ we have $1 + \sin \left( \frac 1 x \right) \ge 0$ so $f(x) \ge 0$. Then $f(\ell_{n})=0$ is a global minimum of $f$, which implies
        $e_{\lambda}f(\ell_{n})=f(\ell_{n})=0$.
        Now, each term of
        \begin{align*}
            f(w) + \frac{(w - x)^2}{2\lambda}
        \end{align*}
        is non-negative so that
        \begin{align*}
            f(w) + \frac{(w- \ell_n)^2}{2\lambda} = 0 &\iff f(w) = 0 \text{ and } (w - \ell_n) = 0 \\
            & \iff w = \ell_n.
        \end{align*}
        Thus,
        \[
        P_\lambda f(\ell_n) = \arg \min_w \left\{ f(w) + \frac{(w-x)^2}{2\lambda}\right\} = \{\ell_n\}.
        \]
    \end{proof}

    \begin{lemma}
        \label{_sequenceDecreasing}
        The sequence $\{\ell_n\}_{n\in\NN}$ is strictly decreasing to $0$.
    \end{lemma}


    \begin{lemma}
    \label{_middleSequence}
        For all $n \ne 0$ we have $\ell_n < m_n < \ell_{n-1}$.
    \end{lemma}
    \begin{proof}
        Let $n \in \NN$ be nonzero.
        We have,
        \begin{align*}
            & \pi(4n + 3) > \pi (4n + 1) > \pi (4n - 1) \\
            \implies &  \frac 2 {\pi (4n +3)} < \frac 2 {\pi (4n +1)} < \frac 2 {\pi (4n - 1)}, \text{ i.e.,} \\
             & \ell_{n} < m_n < \ell_{n-1}.
        \end{align*}
    \end{proof}

    \begin{lemma}
        \label{_sequenceInTermsOfx}
        If $x = \ell_n$ for some $n \in \NN$ then
        \[
        n =  \frac{2 -  3\pi x}{ 4\pi x}.
        \]
        If $x \in (\ell_{n+1}, \ell_{n})$ then
        \[
        n = \left \lfloor \frac{2 - 3\pi x}{ 4\pi x} \right \rfloor.
        \]
    \end{lemma}

    \begin{proof}
        Let $n \in \NN$. Recall
        \[
            \ell_n = \frac{2}{\pi (4n+3)}.
        \]
        First, we see
        \begin{align*}
            x = \ell_n &\iff x = \frac{2}{\pi (4n +3)} \iff \frac{2}{\pi x} = 4n + 3\\
            &\iff n = \frac{2 - 3\pi x}{4 \pi x}.
        \end{align*}
        Now, when $x \ne \ell_k$ for any $k$,
        \begin{align*}
            x \in (\ell_{n+1}, \ell_n) & \iff \ell_{n+1} < x < \ell_n \iff \frac{2}{\pi(4(n+1)+3)} < x < \frac{2}{\pi(4n + 3)} \\
            & \iff n+1 > \frac{2 - 3 \pi x}{4 \pi x} > n \iff \left\lfloor \frac{2 - 3 \pi x}{4 \pi x} \right\rfloor = n.
        \end{align*}
    \end{proof}

    Now, we have done sufficient work to find a gap in $P_\lambda f$.
    \begin{lemma} Let $\lambda>0$. Then
    \label{_gapExistance}
        there exists $N \in \NN$ such that for all $n > N$ we have $m_n \not \in P_\lambda f(\RR)$.
    \end{lemma}

    \begin{proof}
        Recall
        \[
        P_\lambda f(x) = \arg \min_w \left \{ f(w) + \frac{(x-w)^2}{2\lambda} \right \}
        \]
        and for all $y \in P_\lambda f(x)$ we have $e_\lambda f(x) = f(y) + \frac{(x-y)^2}{2\lambda}.$
         Take $N$ large enough so that for any $n > N$ and any $x \in (0, \ell_n)$ we have $n > 5$ and $\frac{1}{2\lambda}x^2 < \frac 4 {5}x$. We claim that $m_{k}\not\in P_{\lambda}f(\RR)$ for $k > N$.
         We have three cases to consider.

        First, suppose $x \le \ell_k$. Then, since $P_\lambda f$ is monotone by Fact~\ref{fact:proxMonotone} and since $P_\lambda f(\ell_n) = \{\ell_n\}$, $y \in P_\lambda f(x)$ has $y \le \ell_k < m_k$.

        Suppose $x \ge \ell_{k-1}$. Then, by monotonicity of $P_\lambda f$ and since $P_\lambda f(\ell_n) = \{\ell_n\}$, $y \in P_\lambda f(x)$ has $y \ge \ell_{k-1} > m_k$.

        Now, suppose $x \in (\ell_k, \ell_{k-1})$. By Lemma \ref{_sequenceInTermsOfx},
        \[
        k - 1 = \left \lfloor \frac{2-3\pi x}{4 \pi x} \right \rfloor.
        \]
        By our assumption on $N$ we have $k-1 \ge 5$. So we have,
        \begin{align*}
            m_k = \frac{2}{\pi(4k+1)} &= \frac{2}{\pi(4(k-1)+5)}\ge \frac{2}{5\pi (k-1)} = \frac{2}{5\pi \left \lfloor \frac{2-3\pi x}{4 \pi x} \right \rfloor} \\
            &\geq \frac{2}{5\pi\left ( \frac{2-3\pi x}{4 \pi x} \right )}
            =\frac{8 x }{5(2 - 3 \pi x)} > \frac{4}{5}x
        \end{align*}
        where the last inequality follows because $x < \ell_N < \ell_0 = \frac{2}{3 \pi}$. Now,
        \begin{align*}
            e_\lambda f(x) \le f(0) + \frac{1}{2\lambda}x^2 = \frac{1}{2\lambda}x^2 < \frac{4}{5}x < m_k < 2 m_k + \frac{(m_k - x)^2}{2 \lambda}
        \end{align*}
        so $m_k \not \in P_\lambda f(x)$. All three cases together give $m_k \not \in P_\lambda f(\RR)$.
\end{proof}

\begin{figure}[h]
    \centering
    \includegraphics[width=0.8\linewidth]{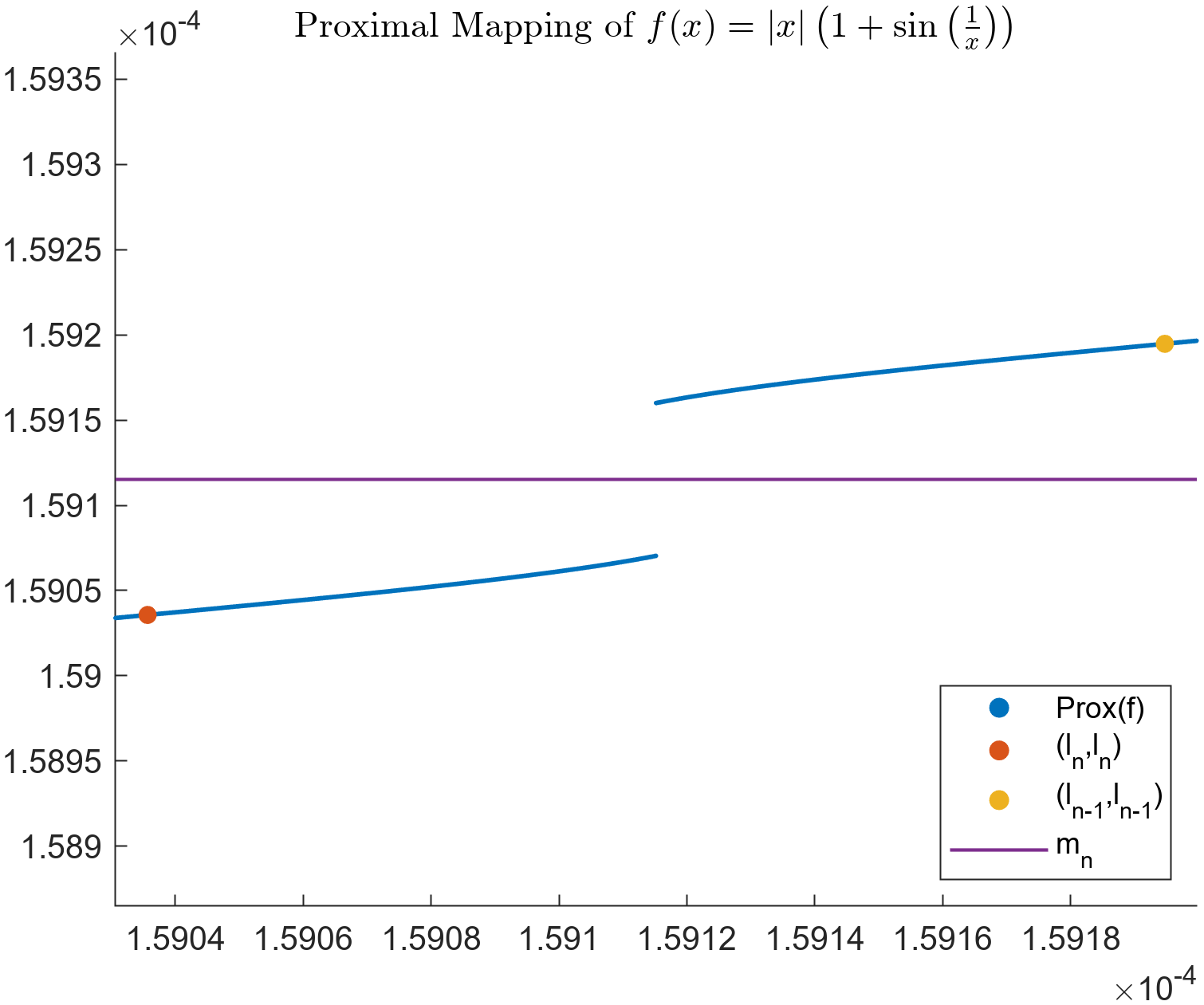}
    \caption{Plot of $P_\lambda f$ between the points $\ell_n$ and $\ell_{n-1}$ with $n = 1000$ and $\lambda = 2^{-37}$. Observe the gap about $m_n$.}
    \label{fig:proxLambda_Zoom_2^(-37)}
\end{figure}

\begin{theorem}
    \label{_Discontinuous}
    Let $f$ be given by \eqref{e:rock-wets}. Then
    in any neighborhood of $0$ there exists a point $\bar x$ such that $P_\lambda f$ is not continuous at $\bar x$ and is multivalued, i.e., $P_{\lambda}f(\bar x)$ contains at least two distinct elements.
    Consequently,  $f$ does not have a locally Lipschitz proximal mapping,
\end{theorem}

\begin{proof}
    Let $O$ be some neighborhood of $0$. Take $N$ large enough so that for all $n > N$ we have, by Lemma \ref{_gapExistance}, $m_n \not \in P_\lambda f(\RR)$ and $\ell_{n-1} \in O$. It is clear that $f$ is proper and lsc. Additionally, since $f(x) \ge 0$, $f$ is prox bounded with threshold $\lambda_f = +\infty$. Thus, by Fact~\ref{fact:proxOscLB} and Fact~\ref{fact:proxMapNonemptyCompact}, $P_\lambda f$ is a set valued mapping from $\RR$ to $\RR$, is osc, locally bounded, and has full domain. Take the two half-spaces $C_1 := \{(x,v)\in \RR\times\RR \mid v > m_n\}$ and $C_2 := \{(x,v)\in\RR\times\RR \mid v < m_n\}$. Then $\bar C_1 \cap C_2 = C_1 \cap \bar C_2 = \varnothing$.
    Put $A:=\gph P_{\lambda}f\cap C_{1}$ and $B:=\gph P_{\lambda}f\cap C_{2}$. Since $m_n \not \in P_\lambda f(\RR)$ we have
    $\gph P_\lambda f = A \cup B$. Moreover, $\{(\ell_{n-1},v) \mid v \in P_\lambda f(\ell_{n-1})\} = \{(\ell_{n-1}, \ell_{n-1})\} \sube A$ and $\{(\ell_{n},v) \mid v \in P_\lambda f(\ell_n)\} = \{(\ell_n,\ell_n)\} \sube B$. Thus, by Corollary \ref{_NecessaryConContRR}, there exists a point $\bar x$ with $\bar x \in [\ell_n, \ell_{n-1}] \sube O$ such that $P_\lambda f$ is not isc at $\bar x$ and $P_\lambda f(\bar x)$ is multivalued.
\end{proof}

\section{Locally Lipschitz proximal mappings do not imply prox-regularity of functions}\label{s:final}
We propose the following general result,
which gives a wealth of examples whose proximal mappings are locally Lipschitz but the functions are not prox-regular.
\begin{theorem} \label{NotProxRegular_Theorem}
    Let $f : \RR \to \RR$ satisfy
\begin{enumerate} [(i)]
    \item\label{i:i1} $f(0) = 0$,
    \item\label{i:i2} there exists $C > 0$ such that $f(x) \geq C|x|$ for every $x\in\RR$, and
    \item\label{i:i3} $f'(x)$ exists on the open set $(0,a)$ for some positive $a$ and oscillates between $-\infty$ and $+\infty$, say
    \[
    \limsup_{x \to 0^+} f'(x) = +\infty \text{ and } \liminf_{x \to 0^+} f'(x) = -\infty.
    \]
\end{enumerate}
Then the following hold:
\begin{enumerate}[(a)]
    \item\label{i:a1} For $\lambda > 0$ and $|x| < C\lambda$ we have $P_\lambda f(x) = \{0\}$ and $e_\lambda f(x) = \frac{1}{2\lambda} x^2$.
    \item\label{i:a2} $[-C,C] \sube \hat \partial f(0)$, and $\partial f(0) = \RR$.
    \item\label{i:a3} $f$ is not prox-regular for $0 \in \partial f(0)$ at $x = 0$.
\end{enumerate}
\end{theorem}

\begin{proof}
    To prove
    \ref{i:a1} it suffices to show
    \[
    f(0) + \frac{1}{2\lambda}x^2 < f(w) + \frac{1}{2 \lambda} (w-x)^2
    \]
    for all $|x| < C\lambda$ and all $w \ne 0$. Suppose $xw \ge 0$. Since $f(w) \geq C|w|$ we have $\frac{f(w)}{|w|}\geq C$. Thus,
    \begin{align*}
        |x| &< C\lambda \leq \frac{\lambda }{|w|}f(w) < \frac{\lambda }{|w|}f(w) + \frac 1 2 |w| \\
        \implies & 0 < \frac{1}{|w|}f(w) + \frac{1}{2 \lambda} (|w| - 2|x|) \\
        \implies & 0 < f(w) + \frac{1}{2 \lambda} (w^2 - 2|wx|) \\
        \implies & \frac{1}{2\lambda}x^2 < f(w) + \frac{1}{2 \lambda} (w^2 - 2wx  + x^2) \\
        \implies & f(0) + \frac{1}{2\lambda}x^2 < f(w) + \frac{1}{2 \lambda} (w-x)^2.
    \end{align*}
    Now, suppose $xw \le 0$,
    \begin{align*}
    f(w) + \frac{1}{2\lambda}(w-x)^2 & = f(w) + \frac{1}{2\lambda}(|x|+|w|)^2 > f(0) + \frac{1}{2\lambda}x^2.
\end{align*}
    In any case we have $P_\lambda f(x) = \{0\}$ for all $|x| < C\lambda$ and
\[
e_\lambda f(x) = f(0) + \frac{1}{2\lambda}x^2 = \frac{1}{2\lambda}x^2.
\]
We now prove result \ref{i:a2}. Recall that $v \in \hat \partial f(\bar x)$ if and only if
\begin{equation}\label{e:barx}
\liminf_{\substack{x \to \bar x \\ x \ne \bar x}} \frac{f(x) - f(\bar x) - \langle v,x - \bar x \rangle}{|x - \bar x|} \ge 0.
\end{equation}
Let $v$ be such that $|v| \le C$. Then, with $\bar x=0$ in \eqref{e:barx},
\begin{align*}
    \liminf_{\substack{x \to 0 \\ x \ne 0}} \frac{f(x) - f(0) - \langle v,x  \rangle}{|x|} &= \liminf_{\substack{x \to 0 \\ x \ne 0}} \frac{f(x) - vx}{|x|} \\
    & \ge \liminf_{\substack{x \to 0 \\ x \ne 0}} \frac{C|x| - vx}{|x|} \\
    &= \liminf_{\substack{x \to 0 \\ x \ne 0}} \left(C- \frac{x}{|x|}v\right) \ge 0.
\end{align*}
Thus, $[-C,C] \sube \hat \partial f(0)$. Now, recall $v \in \partial f(\bar x)$ if and only if there are sequences $x^\nu \xrightarrow[f]{} \bar x$ and $v^\nu \in \hat \partial f(x^\nu)$ with $v^\nu \to v$. Take some $v \in \RR$ and use Darboux's property of $f'$ (see, e.g., \cite[Theorem 7.31]{thomson}) on $(0,a)$ to find $x^\nu \to 0^+$ such that $f'(x^\nu) =y$. Then, by Fact~\ref{fact:diffSubdifEq}, $\hat \partial f(x^\nu) = \{f'(x^\nu)\} = \{y\}$ so $y \in \partial f(0)$. Thus, $\partial f(0) = \RR$.

We now prove result \ref{i:a3}. Recall that $f$ is prox-regular at $\bar x $ for $\bar v$ if $f$ is finite and locally lsc at $\bar x$ with $\bar v \in \partial f(\bar x)$ and there exists $\varepsilon > 0$ and $\rho \ge 0$ such that
\begin{align*}
    f(x') \ge f(x) + \langle v,x' - x \rangle - \frac{\rho}{2}|x'-x|^2 \text{ for all }x' \in \BB_\varepsilon (\bar x) \\
    \text{ when } v \in \partial f(x), |v - \bar v| < \varepsilon, |x-\bar x| < \varepsilon, f(x) < f(\bar x) + \varepsilon.
\end{align*}
By way of contradiction assume that $f$ is prox-regular at $\bar x = 0$ for $\bar v = 0$. Let $\varepsilon > 0$, $\rho \ge 0$, and $x' = 0 \in \BB_\varepsilon(0)$. Apply Darboux's property to find an $x^\nu \to 0^+$ such that $f'(x^\nu) = 0$. Then $\{f'(x^\nu)\} = \{0\} \sube \partial f(x^\nu)$ so we can choose $v = 0$. Choose $\nu$ sufficiently large so that $x^\nu < \min\{\frac{2C}{\rho}, \varepsilon\}$ and $f(x^\nu) < \varepsilon$. We also have $|v - \bar v| = 0 < \varepsilon$, $|x - \bar x| = |x^\nu| < \varepsilon,$ and $f(x) = f(x^\nu) < \varepsilon$. Thus, by our assumption that $f$ is prox-regular,
\begin{align*}
    & f(x') \ge f(x) + \langle v,x' - x \rangle - \frac{\rho}{2}|x'-x|^2 \text{ for all }x' \in \BB_\varepsilon (\bar x) \\
    \implies & f(0) \ge f(x^\nu)  - \frac{\rho}{2}|x^\nu|^2 \\
    \implies & \frac{\rho}{2}|x^\nu|^2 \ge f(x^\nu) \ge C|x^\nu| \\
    \implies & \frac{\rho}{2}|x^\nu| > C
\end{align*}
for this to be true $\rho \ne 0$ and
\[
|x^\nu| = x^\nu > \frac{2C}{\rho}
\]
which contradicts our assumption on $x^\nu$. Thus, $f$ is not prox-regular at $\bar x = 0$ for $\bar v = 0$.
\end{proof}

\noindent Observe that the counter example gone over in Section~\ref{sect:proxExample} does not
meet requirement \ref{i:i2} of Theorem~\ref{NotProxRegular_Theorem}. The following example provides a remedy!

\begin{example}\label{e:explicit} Define $f:\RR\rightarrow\RR$ by
    \begin{equation*}
        x\mapsto
         \begin{cases}
            |x|\left( m + \sin \frac 1 {nx} \right), & x \ne 0, \\
            0, & x = 0
        \end{cases}
    \end{equation*}
    with $m > 1$ and $n \ne 0$. See Figure~\ref{fig:_fmnPlot} below for the plot of the function.
    \begin{figure}[h]
    \centering
    \includegraphics[width=0.8\linewidth]{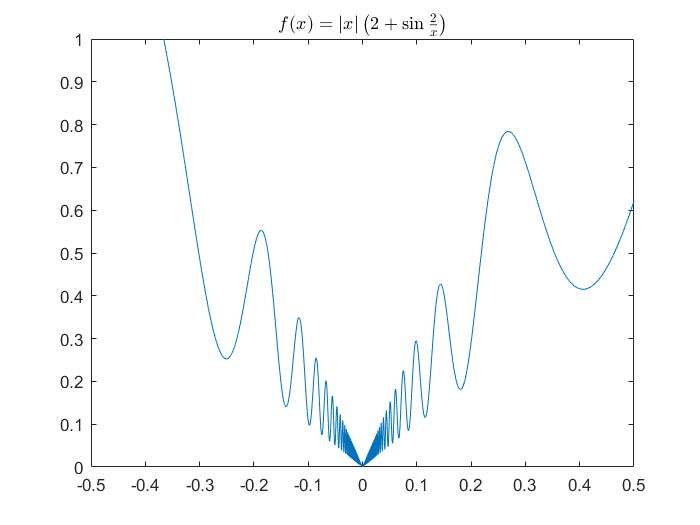}
    \caption{Plot of $|x|\left(m + \sin \left( \frac 1 {nx} \right) \right)$ with $m = 2$ and $n = 1/2$.}
    \label{fig:_fmnPlot}
\end{figure}
Let $C = m -1 > 0$. We have $f(0) = 0$ and $f(x) \ge C|x|$ since $\sin \frac{1}{nx} \ge -1$. Additionally, for $x > 0$ we have
    \[
    f'(x) =
        \displaystyle-\frac{\cos \left(\frac 1 {nx}\right)}{nx} + \sin \left( \frac 1 {nx} \right) + m
    \]
    Let $x^\nu = \frac{1}{2\pi \nu n}$ and $x^\gamma = \frac{1}{n(2\pi \gamma + \pi )}$ so $\cos \left( \frac 1 {nx^\nu}\right) = 1$ and $\cos \left( \frac 1 {nx^\gamma}\right) = -1$. Now,
    \begin{align*}
        \liminf_{x \to 0^+} f'(x) &= \liminf_{x \to 0^+}\left(\displaystyle-\frac{\cos \left(\frac 1 {nx}\right)}{nx} + \sin \left( \frac 1 {nx} \right) + m \right) \\
        &\le \liminf_{\nu \to \infty}\left(\displaystyle-\frac{\cos \left(\frac 1 {nx^\nu}\right)}{nx^\nu} + \sin \left( \frac 1 {nx^\nu} \right) + m \right) \\
        &= \liminf_{\nu \to \infty}\left(\displaystyle-\frac{1}{nx^\nu} + m \right) = - \infty,
    \end{align*}
    and
    \begin{align*}
        \limsup_{x \to 0^+} f'(x) &= \limsup_{x \to 0^+}\left(\displaystyle-\frac{\cos \left(\frac 1 {nx}\right)}{nx} + \sin \left( \frac 1 {nx} \right) + m \right) \\
        & \ge \limsup_{\gamma \to \infty}\left(\displaystyle-\frac{\cos \left(\frac 1 {nx^\gamma}\right)}{nx^\gamma} + \sin \left( \frac 1 {nx^\gamma} \right) + m \right) \\
        &= \limsup_{\gamma \to \infty}\left(\displaystyle\frac{1}{nx^\gamma} +  m \right) = + \infty.
    \end{align*}
    Therefore the requirements for Theorem \ref{NotProxRegular_Theorem} are met and we can conclude that
    \begin{enumerate}[(a)]
    \item for $|x| < C\lambda$, $P_\lambda f(x) = \{0\}$, and $e_\lambda f(x) = \frac{1}{2\lambda} x^2$.
    Consequently, $P_{\lambda}f$ is Lipschitz and $e_{\lambda}f$ is differentiable around $0$.
    \item $[-C,C] \sube \hat \partial f(0)$, and $\partial f(0) = \RR$.
    \item $f$ is not prox-regular for at $x = 0$ for $0 \in \partial f(0)$,
\end{enumerate}
\end{example}


Hence, Example~\ref{e:explicit} shows that Fact~\ref{_rockafellarProp}\ref{i:mor1} and \ref{i:mor2} do not suffice for the prox-regularity of the function.

\section*{Acknowledgements}
Isaac Jasper was supported by NSERC USRA and Xianfu Wang was partially supported by NSERC Discovery Grant.

\section*{Declarations}
The authors have no competing interests to declare that are relevant to the content of this article.

\bibliography{ref}

\end{document}